\documentclass[a4paper, 10pt, twoside]{article}

\usepackage{amsmath, amscd, amsfonts, amssymb, amsthm, latexsym, url, color, todonotes, bm, framed, rotating, enumerate} 
\usepackage{graphicx}
\usepackage[left=1in, right=1in, top=1.2in, bottom=1in, includefoot, headheight=13.6pt]{geometry}
\usepackage{booktabs}
\usepackage{adjustbox}
\usepackage{mathtools}
\input{xy}
\xyoption{all}

\usepackage[colorlinks,breaklinks=true]{hyperref}
\usepackage[figure,table]{hypcap}
\hypersetup{
	bookmarksnumbered,
	pdfstartview={FitH},
	citecolor={black},
	linkcolor={black},
	urlcolor={black},
	pdfpagemode={UseOutlines}
}
\makeatletter
\newcommand\org@hypertarget{}
\let\org@hypertarget\hypertarget
\renewcommand\hypertarget[2]{%
  \Hy@raisedlink{\org@hypertarget{#1}{}}#2%
} 
\makeatother 

\newtheorem{theorem}{Theorem}[section]

\newtheorem{corollary}[theorem]{Corollary}
\newtheorem{proposition}[theorem]{Proposition}

\theoremstyle{definition}
\newtheorem{definition}[theorem]{Definition}
\newtheorem{remark}[theorem]{Remark}

\newtheorem{conjecture}[theorem]{Conjecture}
\newtheorem{question}[theorem]{Question}

\newcommand{\xysquare}[8]{
\[\xymatrix{
#1 \ar@{#5}[r] \ar@{#6}[d] & #2 \ar@{#7}[d]\\
#3 \ar@{#8}[r] & #4
}\]
}

\newcommand{\bb}{\mathbb}

\newcommand{\comment}[1]{}

\renewcommand{\phi}{\varphi}

\newcommand{\roi}{\mathcal{O}}

\newcommand{\sub}[1]{{\mbox{\rm \scriptsize #1}}}

\newcommand{\To}{\longrightarrow}

\newcommand{\xto}{\xrightarrow}

\renewcommand{\cal}{\mathcal}
\renewcommand{\hat}{\widehat}
\renewcommand{\frak}{\mathfrak}

\renewcommand{\tilde}{\widetilde}

\renewcommand{\ker}{\operatorname{Ker}}
\renewcommand{\projlim}{\varprojlim}

\DeclareMathOperator{\dlog}{dlog}

\DeclareMathOperator{\Gal}{Gal}

\DeclareMathOperator{\Hom}{Hom}

\DeclareMathOperator{\Spec}{Spec}

\DeclareMathOperator{\SK}{SK}
\DeclareMathOperator{\ab}{ab}
\DeclareMathOperator{\rec}{rec}

\newcommand{\CH}{C\!H}

\newcommand{\xTo}[1]{\stackrel{#1}{\To}}

\newcommand{\rcoeq}[3]{\xymatrix{ #1\ar@/^3mm/[r]^f \ar@/_3mm/[r]_g & #2 \ar[l]_e\ar[r] & #3}}

\usepackage[bbgreekl]{mathbbol}
\DeclareSymbolFontAlphabet{\mathbbm}{bbold}

\usepackage{fancyhdr}

\usepackage{sectsty}
\sectionfont{\Large\sc\centering}
\chapterfont{\large\sc\centering}
\chaptertitlefont{\LARGE\centering}
\partfont{\centering}

\begin{document}
\itemsep0pt

\title{
Higher dimensional local class field theory}

\author{Morten L\"uders}

\date{}

\maketitle

\begin{abstract}
In this largely expository article we give an introduction to higher dimensional local class field theory, more precisely class field theory for smooth projective schemes over local fields. 
We begin with a quick summary of some of the main results of classical local class field theory. Then we explain the main definitions and statements in the higher dimensional setting. We give a new and simple proof of the $\ell$-part, which is originally due to Jannsen--Saito and Forr\'e. Inspired by this proof, we propose a conjecture which would imply the remaining open cases with $p$-coefficients and which improves our understanding of the $p$-part. A particular goal of the exposition is to outline analogies between the classical and the higher dimensional theory and on how the unit filtration and ramification, i.e. the $p$- and the $\ell$-part, can be interpreted in the higher dimensional setting. 
\end{abstract}
\begingroup

\smallskip

\par
\endgroup

\section{Introduction}
Let $K$ be a local field of characteristic zero, i.e. a finite extension of $\bb Q_p$. Local class field theory is one of the great mathematical achievements of the first half of the 20th century due to Artin, Hilbert, Tagaki, Tate to name just some of the most important names. It relates abelian extensions of $K$, that is objects lying over $K$, to subgroups of $K^\times$, that is to information that is more intrinsic to $K$. To be more precise, in local class field theory a so called \textit{reciprocity map} 
$$\rec_{K}: K^\times\To \Gal(\bar{K}/K)^{\rm ab}$$ 
and a one-to-one correspondence between closed subgroups of finite index of $K^\times$ and abelian extensions of $K$ are constructed. 
In the second half of the 20th century, initiated by Bloch, Colliot-Thélène, Kato, Lang, Saito, Sansuc and Soulé, similar reciprocity maps have been defined and studied for schemes over $\bb F_p,\bb Q_p$ and $\bb Z$ replacing the Galois group by the \'etale fundamental group and the unit group by a higher Chow group whose bidegree depends on the base field (see for example \cite{Bloch1981, CSS83, Ka86, KaS83, KaS86, Lang1956, Saito1985}). This study has been continued in the 21st century in particular by Forré, Jannsen, Kerz, Saito, Schmidt and Wiesend and, at least in the smooth projective case, the only remaining open problem is to understand the $p$-part for schemes of dimension $>2$ over local fields (see for example \cite{Fo15, JS03, KeS16, Schmidt05, Wiesend07}). A nice exposition to higher dimensional class field theory over $\bb F_p$ and $\bb Z$ can be found in \cite{Sz10}.

The goal of this article is to work out the precise analogies between higher local and classical local class field theory. A particular focus will be on which part of the Chow group of a smooth projective scheme $X_K$ over $K$ captures the information about abelian unramified, tamely ramified and wildely ramified covers of $X_K$ which is contained in the \'etale fundamental group. In the course of this discussion we will give a simple proof, if we admit some black boxes concerning base change and higher dimensional class field theory over $\bb F_p$, of the $\ell$-part of the following theorem, which is due to Jannsen--Saito in dimension $\leq 2$ (see \cite{JS03}) and Forré in arbitrary dimension (see \cite{Fo15}).
\begin{theorem}\label{theorem_main_intro} 
Let $K$ be a local field with residue field $k$ of characteristic $p$ and let $X_K$ be a smooth projective $K$-scheme of dimension $d$ with good reduction. Then the reciprocity map
$$\rec_{X_K}:\CH^{d+1}(X_K,1) \to \pi_1^{\ab}(X_K)$$
is an isomorphism mod $\ell\in \bb N_{>0}$ for all $(\ell,p)=1$ (resp. mod $\ell$ for all $\ell\in \bb N_{>0}$ if $\dim X_K=d=2$). 
\end{theorem}
Furthermore, we propose a conjectural relation of $\CH^{d+1}(X_K,1)_{\bb Z/p^r\bb Z}$ to a version of this Chow group on the (thickened) special fiber. This conjecture would imply that $\rec_{X_K}^{\bb Z/p^r\bb Z}$ is an isomorphism for $d\geq 3$ and gives an interpretation of the unit filtration in the higher dimensional setting.

In Section \ref{section_cft} we recall some of the main results of classical local class field theory. In Section \ref{section_hcft} we first define the main players in the higher dimensional case. Then we turn to the above described goal. We summarise the analogies between the higher and classical theory in Table \ref{tab:PECs}. The author claims no  originality for the results of this article. 
Due to the large amount of literature on class field theory, we do not attempt to give a complete account of all results and works on the topic but focus on some of the main ideas.

\paragraph{Acknowledgements.} I would like to thank Immanuel Klevesath and Nils Witt for a careful reading and very helpful feedback. The author is funded by the Deutsche Forschungsgemeinschaft (DFG, German Research Foundation), project number $557768455$ and by the Deutsche Forschungsgemeinschaft (DFG, German Research Foundation)
TRR 326 \textit{Geometry and Arithmetic of Uniformized Structures}, project number 444845124.

\section{Classical local class field theory}\label{section_cft}
In this section let $K$ be a local field of characteristic zero, i.e. a finite extension of $\bb Q_p$. Let $\roi_K$ be the ring of integers of $K$ with maximal ideal $\mathfrak{p}$ and residue field $k$ of characteristic $p$. 
There are many excellent books on classical local class field theory; see for example \cite{NeukirchKkt}.
The main theorem of local class field theory is the following: 
\begin{theorem}\label{theorem_main_local_cft}
\begin{enumerate}
\item There exists a unique homomorphism, called Artin's local reciprocity map,
$$\rec_{K}: K^\times\To \Gal(\bar{K}/K)^{\rm ab}$$
which for every finite abelian Galois extension $L/K$ induces an isomorphism
$$\rec_{L/K}:K^\times/\mathrm{Nm}_{L/K}L^\times \xTo{\cong} \Gal(L/K)^{}$$
and such that for every prime element $\pi\in K$ and every finite unramified extension $L$ of $K$, $\rec_{L/K}(\pi)=\mathrm{Frob}_{L/K}$.
\item The association $L\mapsto \mathrm{Nm}_{L/K}L^\times\subset K^\times$ is an inclusion reversing bijection from the set of abelian extensions of $K$ to the set of norm subgroups of $K^\times$. Furthermore, the norm subgroups of $K^\times$ are exactly the open (and therefore closed) subgroups of finite index of $K^\times$ ("Tagaki's existence theorem'').
\end{enumerate}
\end{theorem}

\begin{remark}\label{remark_CFT_finite_field}
The reciprocity map $\rec_{\bb F_p}$ for a finite field $\bb F_p$ is given by the natural inclusion $
\bb Z\to \hat{\bb Z}$ and corresponds to the unramified extensions in Theorem \ref{theorem_main_local_cft}.
\end{remark}

One possible construction of $\rec_{L/K}$ is through the following perfect pairing, called Nakayama--Tate duality, of Tate cohomology groups
$$\langle -,-\rangle_{}:\hat{H}^0(\Gal(L/K),\Hom(\bb Z,L^\times))\times H^2(\Gal(L/K),\bb Z)\xto{\cup}H^2(\Gal(L/K),L^\times)\cong \frac{1}{\sharp G}\bb Z/\bb Z.$$
Since $$\hat{H}^0(\Gal(L/K),\Hom(\bb Z,L^\times))\cong K^\times/N_{L/K}L^\times$$ and $$H^2(\Gal(L/K),\bb Z)^*\cong H^1(\Gal(L/K),\bb Q/\bb Z)^*\cong \Gal(L/K)^{\rm ab},$$
where ${}^*$ denotes the Pontrjagin dual, we get in particular a morphism $\rec_{L/K}:K^\times \xTo{} \Gal(L/K)^{\rm ab}.$ Using the notion of a \textit{class formation}, $\rec_{L/K}$ can be used to construct $\rec_K$.

In order to get a more refined understanding of both sides of the reciprocity map one defines filtrations on the domain and codomain of $\rec_K$. The right filtration on the domain is the unit filtration:
\begin{definition} Let
$$\dots \subset 1+\mathfrak{p}^2\roi_K\subset1+\mathfrak{p}\roi_K\subset \roi_K^\times\subset K^\times$$
be the \textit{unit filtration} and set $U^{(0)}:= \roi_K^\times$ and $U^{(i)}:= 1+\mathfrak{p}^i\roi_K$ for $i\geq 1$. 
\end{definition}
The unit filtration, and its graded pieces, have the following properties: firstly, there is a short exact sequence \begin{equation}\label{valuation_sequence}
0\To U^{(0)}\To K^\times \xTo{v} \bb Z\To 0,
\end{equation}
in which $v$ is the $\frak p$-adic valuation.
Secondly, the quotient $U^{(0)}/U^{(1)}$ is of order prime to $p$ and thirdly the quotients $U^{(i)}/U^{(i+1)}$ are $p$-groups for $i\geq 1$. In particular, $U^{(1)}$ is the unique $p$-Sylow subgroup of $U^{(0)}$. 

On the target of $\rec_K$ one can define two different and important filtrations. One in the \textit{lower} and one in the \textit{upper numbering}, the latter one corresponding to the unit filtration under the reciprocity map. Since the definition of the filtration in the upper numbering is complicated we just give the first three pieces; these suffice for our purposes in the next section. 
\begin{definition} Let $L$ be a Galois extension of $K$. Let $$G^{(-1)}=\Gal(L/K),$$ 
$$G^{(0)}=\mathrm{I}(L/K),$$ 
$$G^{(1)}=\mathrm{R}(L/K):=\{\sigma\in \Gal(L/K)\;|\;\frac{\sigma x}{x}=1 \;\text{mod}\; \mathfrak{p}_L\; \forall x\in L^\times\}$$ be the Galois group, the \textit{inertia} group and the \textit{ramification} group.
\end{definition} 
Under the Galois correspondence, $G^{(0)}=\mathrm{I}(L/K)$ corresponds to the maximal unramified extension $L^{\mathrm{I}(L/K)}$ of $K$ in $L$
and $G^{(1)}=\mathrm{R}(L/K)$ to the maximal tamely ramified extension $L^\mathrm{R(L/K)}$ of $K$ in $L$. The unit and the ramification filtration are related via the following theorem.
\begin{theorem} Let $K^{\mathrm{ab}}$ be the maximal abelian  Galois extension of $K$. $G^{(1)}=\mathrm{R}(K^{\mathrm{ab}}/K)$ is the unique $p$-Sylow subgroup of $G^{(0)}=\mathrm{I}(K^{\mathrm{ab}}/K)$. 
Then the reciprocity map induces isomorphisms $$\rec_{K}|_{U^{(0)}}:U^{(0)}\xTo{\cong}\mathrm{I}(K^{\mathrm{ab}}/K)$$ 
and
$$\rec_{K}|_{U^{(1)}}:U^{(1)}\xTo{\cong}\mathrm{R}(K^{\mathrm{ab}}/K).$$
\end{theorem}
In particular, we get isomorphisms $$K^\times/\ell \xTo{\cong} \Gal(K^{\mathrm{ab}}/K)/\ell\cong \bb Z/\ell\oplus (\mathrm{I}/\mathrm{R})/\ell$$ for $(\ell,p)=1$ and $$K^\times/p^r \xTo{\cong} \Gal(K^{\mathrm{ab}}/K)/p^r\cong \bb Z/p^r\oplus \mathrm{R}/p^r. $$ We restate this as a Corollary to emphasise that Theorem \ref{theorem_main_intro} is a higher dimensional generalisation of this version of local class field theory. 

\begin{corollary}
The reciprocity map $\rec_{K}^{\bb Z/\ell\bb Z}: K^\times/\ell\To \Gal(\bar{K}/K)^{\rm ab}/\ell$ is an isomorphism for all $\ell \in \bb N_{>0}$.
\end{corollary}

\section{Class field theory for proper schemes over local fields}\label{section_hcft}
\subsection{The main players}
Even though we just focus on finite and local fields, for the sake of conceptualisation and understanding the bigger picture, it is useful to mention higher local fields.
\begin{definition}
An $n$-\textit{dimensional local field} (or higher local field) is a sequence of fields 
\[ K_0, K_1, \dots, K_n = K \]
such that:
\begin{enumerate}
    \item $K_0$ is a finite field.
    \item For each $i \in \{0, 1, \dots, n-1\}$, the field $K_{i+1}$ is a complete discrete valuation field whose residue field is $K_i$.
\end{enumerate}
\end{definition}
For example, a $0$-local field is a finite field, $1$-local field is a local field and a $2$-local field is a complete discretely valued field whose residue field is a local field. As we have seen, in the finite field case the source of the reciprocity map are the integers and in the local field case the units. The common generalisation needed for higher local fields is that of Milnor $K$-theory.
\begin{definition}\label{definition_KM}
Let $R$ be a (always commutative) ring. We define the $j^\sub{th}$ \textit{Milnor K-group} $K^M_j(R)$ to be the quotient of $(R^{\times})^{\otimes j}$ by the Steinberg relations, i.e. the subgroup of $(R^{\times})^{\otimes j}$ generated by elements of the form $a_1\otimes\cdots\otimes a_j$ where $a_l+a_k=1$ for some $1\leq l<k\leq j$. As usual, the image of $a_1\otimes\cdots\otimes a_j$ in $K^M_j(R)$ is denoted by $\{a_1,\dots,a_j\}$.
\end{definition}


Then Kato has shown \cite{Kato1979,Kato1980} that for an $N$-local field $K$ there is a reciprocity map
$$\rec_K: K^M_N(K)\to \Gal_{}({K}^{\ab}/K)$$
generalising the reciprocity map in Section \ref{section_cft}. We now turn to defining the reciprocity map for schemes. For this we are going to first introduce the right source of the map. This was defined by Bloch in \cite{Bloch1981}:

\begin{definition} Let $K$ be an $N$-local field and $X_K$ be a $K$-scheme of dimension $d=\dim(X_K)$. We set
$$\SK_N(X_K):=\mathrm{coker}[\bigoplus_{x\in X_K^{(d-1)}}K^M_{N+1}(k(x))\xTo{\partial} \bigoplus_{x\in X_K^{(d)}} K^M_N(k(x))],$$
where $X_K^{(i)}$ denotes the set of points of codimension $i$. Here $\partial$ is defined using the so-called \textit{tame symbol}, which is a generalisation of the valuation in (\ref{valuation_sequence}). Indeed, if $x\in X_K^{(d-1)}$, then every point of codimension $1$ one on the normalisation of the curve $\overline{\{x\}}$ defines a discrete valuation, which combined with the norm map and tame symbol on Milnor K-theory allows to define $\partial$. For more details we refer to \cite[Sec. 7]{GilleSzamuely2006}. 
\end{definition}

We will use a different interpretation of these groups in terms of higher Chow groups, also defined by Bloch in \cite{Bl86}:
\begin{definition} Let $k$ be a field and $X_k$ a finite type $k$-scheme. Bloch's \textit{higher Chow groups} are defined as the homology groups of a certain complex
$$\dots\to z^*(X_k,n+1)\xrightarrow{\partial_{n+1}} z^*(X_k,n)\xrightarrow{\partial_n} z^*(X_k,n-1)\xrightarrow{\partial_{n-1}}\dots$$
Here $\Delta^n:= \Spec(k[t_0,...,t_n]/\sum t_i-1)$ and $z^*(X_k,n)\subset z^*(X_k\times \Delta^n)$ is the subgroup of cycles meeting all faces of $X_k\times \Delta^n$ properly (for more details see \textit{loc. cit.}). We denote $$\CH^*(X_k,n):=\ker[ z^*(X_k,n)\xrightarrow{\partial_{n}} z^*(X,n-1)]/\text{im}[z^*(X_k,n+1)\xrightarrow{\partial_{n+1}} z^*(X,n)].$$
If $\Lambda=\bb Z/m\bb Z,\; m\in \bb N$, then we write $\CH^{*}(X_k,n)_{\Lambda}$ for the same definition but with $\Lambda$-coefficients.
\end{definition}
The following proposition relates the previous two definitions. For the proof see for example \cite[Prop. 2.3]{Lu17}. The key ingredients are a local to global spectral sequence for higher Chow groups (see \cite[Sec. 10]{Bl86}) and the isomorphisms $K^M_N(k(x))\xto{\cong}\CH^N(\Spec k(x),N)$ (see \cite{To92}).  
\begin{proposition}
$SK_N(X_K)\cong\CH^{d+N}(X_K,N).$\footnote{Due to the definition of $SK_N$ one also calls higher Chow groups in this bidegree \textit{Chow group of zero cycles with coefficients on Milnor K-theory}.}
\end{proposition}

We now return to the situation of Section \ref{section_cft}. Let $K$ be a local field, $\roi_K$ the ring of integers of $K$ with maximal ideal $\mathfrak{p}$, uniformizer $\pi$ and residue field $k$ of characteristic $p$. Let $X$ be a smooth projective $\Spec\roi_K$-scheme of relative dimension $d$ with generic fiber $X_K$ and special fiber $X_k$. The groups we are interested in are
$$
SK_1(X_K)\cong\CH^{d+1}(X_K,1)$$ and $$SK_0(X_k)\cong\CH^{d}(X_k).
$$
For higher Chow groups we have a long exact localisation sequence (see \cite{Le01}) at our disposal which allows us to relate the two. We consider the following part of it:
\begin{equation}\label{higher_valuation_exact_sequence}\CH^{d+1}(X,1)_{\Lambda}\To \CH^{d+1}(X_K,1)_{\Lambda}\To \CH^{d}(X_k)_{\Lambda} \To \CH^{d+1}(X)_{\Lambda}=0 \end{equation}
Note that the group $\CH^{d+1}(X)=0$ because it is generated by closed points on the special fiber which can be lifted to smooth curves on $X$ and which are divisors of $\pi$ on these curves.
For $d=0$, this sequence reduces to the valuation sequence (\ref{valuation_sequence}). 
Our line of interpretation will therefore be that $\CH^{d}(X_k)$ captures the unramified covers of $X_K$ and $\CH^{d+1}(X,1)_{}$ the ramified covers. Furthermore, letting $\ell$ be an integer prime to $p$, we will interpret $\CH^{d+1}(X,1)_{\bb Z/\ell\bb Z}$ as capturing the tame and $\CH^{d+1}(X,1)_{\bb Z/p^r\bb Z}$ the wild part.
But before we continue, we need to recall the right replacement of the Galois group in the setting of schemes. We will be very brief. Let $X$ be a connected scheme and $\bar{x}$ a geometric point. The usual topological fundamental group can be identified with the automorphism group of a universal cover. This motivated Grothendieck to define the fiber functor
$$F:\{\text{finite \'etale covers of }X\}\to \mathrm{Sets}, Y\mapsto\Hom_X(\bar{x},Y)$$ 
and to consider the projective system $(X_i)_{i\in I}$ of Galois covers $X_i$ of $X$ which prorepresent $F$ as the universal covering space of $X$. Then one defines
$$\pi_1(X,\bar{x}):=\projlim \mathrm{Aut}_X(X_i).$$
One can show that for a field $k$ this generalises the Galois group, i.e. $\pi_1(\Spec k)\cong \Gal(\bar{k}/ k)$.

We will work with the following cohomological expression of the \'etale fundamental group with finite coefficients. In the following all cohomology groups will be \'etale cohomology groups. The formalism of \'etale cohomology will enable us to use localisation sequences and proper base change on the \'etale side of the comparison.
\begin{proposition}\label{proposition_comparison_pi_coh}
There are isomorphisms
$$ \pi_1^{\ab}(X_k)/\ell\cong \Hom(H^1(X_k,\bb Z/\ell),\bb Q/\bb Z)\cong H^{2d}(X_k,\mu_\ell^{\otimes d})$$
and 
$$ \pi_1^{\ab}(X_k)/p^r\cong \Hom(H^1(X_k,\bb Z/p^r),\bb Q/\bb Z)\cong H^{2d}(X_k,\nu_r(d)[-d])$$
for the special fiber. Here $\nu_r(d)$ are so called \textit{logarithmic de Rham Witt sheaves}, which are $p$-torsion analogues of $\mu_\ell^{\otimes d}$ in characteristic $p$ (see \cite{Milne1986}). For the generic fiber there are isomorphisms
$$ \pi_1^{\ab}(X_K)/\ell\cong \Hom(H^1(X_K,\bb Z/\ell),\bb Q/\bb Z)\cong H^{2d+1}(X_K,\mu_\ell^{\otimes d+1}),$$
\end{proposition}
\begin{proof}
The isomorphisms on the right are Poincar\'e duality due to \cite[Exp. XVIII]{SGA4} and \cite[Cor. 1.12]{Milne1986} with $\ell$ and $p$ coefficients respectively. The isomorphisms on the left come from the fact that $H^1$ classifies $A$-torsors for a finite abelian group $A$, which correspond to continuous homomorphsims $\pi_1(X)\to A$, and Pontryagin duality.
\end{proof}

\subsection{The reciprocity map.}
Let $K$ be a local field, $\roi_K$ the ring of integers of $K$ with maximal ideal $\mathfrak{p}$ and residue field $k$ of characteristic $p$.
\begin{theorem} 
Let $X_K$ be proper over $K$, then the maps $K^M_1(k(x))\to \pi_1^{\ab}(\Spec k(x))\to \pi_1^{\ab}(X_K)$ defined by local class field theory induce a map
$$\mathrm{rec_{X_K}}:\SK_1(X_K)\to \pi_1^{\ab}(X_K)$$
called the reciprocity map of $X_K$. Similarly, let $X_k$ be proper over $k$, then the maps $K^M_0(k(x))\to \pi_1^{\ab}(\Spec k(x))\to \pi_1^{\ab}(X_k)$ defined in Remark \ref{remark_CFT_finite_field} induce a map
$$\mathrm{rec_{X_k}}:\SK_0(X_k)\to \pi_1^{\ab}(X_k)$$
called the reciprocity map of $X_k$.
\end{theorem}
\begin{proof}
Replacing $X_K$ by $\overline{\{x\}}$ for $x\in X_K^{(d-1)}$, we may assume that $\dim X_K=1$. Let $\tilde{X}_K$ be the normalisation of $X_K$. Then by the definition of $\partial$ there is a commutative diagram
$$\xymatrix{
K_2^M(k(\tilde{X}_K)) \ar@{=}[d] \ar[r]^-{\partial}  & \bigoplus_{x\in \tilde{X}_K^{(d)}} k(x)^\times \ar[d]^{\rm Norm} \ar[r]^{}  & \pi_1^{\ab}(\tilde{X}_K) \ar[d] \\
K_2^M(k(X_K)) \ar[r]^-{\partial}  &  \bigoplus_{x\in {X}_K^{(d)}} k(x)^\times \ar[r]^{}  & \pi_1^{\ab}(X_K) \
}$$
which reduces us to showing that the upper row is a complex (cf. \cite[Sec. 2 and 3]{SaitoUnramifiedCFTarithSch}). This is shown in \cite[Ch. 1]{Saito1985}. The same proof works in the second case, where instead of reducing to the case of curves over local fields one reduces to the case of a curve over the integers which is treated in classical class field theory. 

\end{proof}

\subsection{The $0$-local field case: finite fields}

Let $X_k$ be a smooth projective scheme over a finite field $k$. Then the class field theory for finite fields, i.e. the class formation $\bb Z\to \Gal(\overline{k}/k)\cong \hat{\bb Z}$ generalises to the following picture: 
\begin{theorem}\label{theorem_CFT_finite_fields} 
There is a commutative diagram with exact rows
$$\xymatrix{
0\ar[r]^{}  & \CH^{d}(X_k)_0 \ar[d]^{\mathrm{rec}^0_{X_k}}_{\cong} \ar[r] & SK_0(X_k)=\CH^{d}(X_k) \ar[d]^{\mathrm{rec}_{X_k}} \ar[r]^-{\deg} & \bb Z \ar@{^{(}->}[d] \ar[r]^{}  & 0 \\
0\ar[r]^{}  & \pi_1^{\ab}(X_k)_0 \ar[r]^{}  & \pi_1^{\ab}(X_k)\ar[r]^{}  & \Gal(\bar{k}/k)\cong \hat{\bb Z}\ar[r]^{}  & 0, \
}$$
in which the first vertical map is an isomorphism.
\end{theorem} 
\begin{proof}
The surjectivity of $\mathrm{rec^0_{X_k}}$ is proved by Lang in \cite{Lang1956}, \cite[p. 405]{LangBulletin}. The injectivity of $\mathrm{rec^0_{X_k}}$ is proved in two different ways in \cite[Thm. 1]{KaS83} and \cite{CSS83} first for surfaces and then reduced to the surface case by a Bertini argument and a Lefschetz hyperplane argument for the \'etale fundamental group. 
\end{proof}

\subsection{The $1$-local field case: the $\ell$-part.}
We turn to the main objective of the article and fix the following notation for the rest of the article: let $K$ be a local field, $\roi_K$ the ring of integers of $K$ with maximal ideal $\mathfrak{p}$ and residue field $k$ of characteristic $p$. Let $X$ be a smooth projective $\Spec\roi_K$-scheme of relative dimension $d$ with generic fiber $X_K$ and special fiber $X_k$.
\begin{theorem}\label{theorem_ell_adic_case}
Let $\ell\in \bb N$ with $(\ell, p)=1$ and $\Lambda=\bb Z/\ell \bb Z$. Then the mod-$\ell$-reciprocity map
$$\rec_{X_K}^{\bb Z/\ell\bb Z}:\CH^{d+1}(X_K,1)_{\Lambda} \to \pi_1^{\ab}(X_K)/\ell$$
is an isomorphism.
\end{theorem}
\begin{proof}
We consider the commutative diagram with exact rows
$$\xymatrix{
\CH^{d}(X_k,1)_{\Lambda}\ar[r]^{} \ar[d]^\cong_{} & \CH^{d+1}(X,1)_{\Lambda}\ar[d]_{\rho_{X,\ell}^{2d+1,d+1}} \ar[r] & \CH^{d+1}(X_K,1)_{\Lambda} \ar[d]_{\rec_{X_K}^{\bb Z/\ell\bb Z}} \ar[r] & \CH^{d}(X_k)_{\Lambda} \ar[d]^\cong_{\rec_{X_k}^{\bb Z/\ell\bb Z}} \ar[r]^{}  & 0 \\
 H^{2d-1}_{}(X_k,\mu_\ell^{\otimes d})\ar[r]^{}  & H^{2d+1}(X,\mu_\ell^{\otimes d+1}) \ar[r]^{}  & H^{2d+1}(X_K,\mu_\ell^{\otimes d+1}) \ar[r]^{}  &   H^{2d}_{}(X_k,\mu_\ell^{\otimes d}) \ar[r]^{}  & 0. \
}$$
The upper sequence is the localization sequence for higher Chow groups, the lower one for \'etale cohomology combined with the Gysin isomorphisms $H^{2d+1}_{X_k}(X,\mu_\ell^{\otimes d+1})\cong H^{2d-1}_{}(X_k,\mu_\ell^{\otimes d})$ and $H^{2d+2}_{X_k}(X,\mu_\ell^{\otimes d+1})\cong H^{2d}_{}(X_k,\mu_\ell^{\otimes d})$. The map $\rec_{X_k}^{\bb Z/\ell\bb Z}$ is an isomorphism by Theorem \ref{theorem_CFT_finite_fields} and Prop. \ref{proposition_comparison_pi_coh}. The map
left vertical map is an isomorphism by the Kato conjectures over finite fields \cite[Conj. 0.3]{Ka86}. Indeed, the Kato conjectures over finite fields imply more generally that there are isomorphisms $\CH^d(X_k,j)_\Lambda\xto{\cong}H^{2d-j}(X_k,\mu_\ell^{\otimes d})$ for all $j\geq 0$. The case $j=0$ is implied by higher dimensional class field theory, the other cases have been proved in joined efforts by Jannsen, Kerz and Saito (see for example \cite{KeS12}). According to our philosophy to use information on the special fiber and base change we use these results as a black box. It should be noted that the cases $j=0,1$ are milder than the cases $j\geq 2$.
 
It remains to show that $\rho_{X,\ell}^{2d+1,d+1}$ is an isomorphism, and indeed, there is a commutative diagram of isomorphisms
$$\xymatrix{
\CH^{d+1}(X,1)_{\Lambda} \ar[d]^-{\cong}_{\rho_{X,\ell}^{2d+1,d+1}} \ar[r]^{res^{d+1,1}}_-{\cong} & \CH^{d+1}(X_k,1)_{\Lambda} \ar[d]^-{\cong}_{\rho_{X_k,\ell}^{2d+1,d+1}}  \\
H^{2d+1}(X,\mu_\ell^{\otimes d+1}) \ar[r]_-{\cong}  & H^{2d+1}(X_k,\mu_\ell^{\otimes d+1}).  \
}$$
The lower horizontal map is an isomorphism by proper base change for \'etale cohomology. The map $\rho_{X_k,\ell}^{2d+1,d+1}$ can easily be seen to be an isomorphism using the coniveau spectral sequence for \'etale cohomology and cohomological dimension. For the fact that $\rho_{X_k,\ell}^{2d+1,d}$ is an isomorphism 
we use the fact that $res^{d+1,1}$ is an isomorphism more generally for any smooth projective scheme over a henselian discrete discrete valuation ring \cite{Lu17}. The latter result is a Zariski analogue of proper base change, i.e. a proper base change theorem for motivic cohomology, and can be proved by purely geometric techniques using Bertini theorems and a reduction to relative curves. This is our second black box.
\end{proof}

\subsection{The $1$-local field case: the $p$-part, summary and open problems}
The story concerning the $p$-part is more difficult. We keep the notation of the previous section. Furthermore, we denote by $X_n:=X\times_{\Spec \roi_K}\Spec \roi_K/\pi^n$, where $\pi\in \roi_K$ is a uniformizer. In particular $X_1=X_k$.
\begin{theorem}\label{theorem_JS_p} \cite[Thm. 1.8(3)]{JS03}
Let $d=\dim X_K=2.$ Then the mod-$p^r$-reciprocity map
$$\rec_{X_K}^{\bb Z/p^r\bb Z}:\CH^{d+1}(X_K,1)_{\bb Z/p^r\bb Z}\to H^{2d+1}(X_K,\mu_{p^r}^{\otimes d+1})= \pi_1^{\ab}(X_K)/p^r$$
is an isomorphism for all $r\geq 1$.
\end{theorem}

The proof of loc.cit. involves the Kato conjectures with $p$-coefficients and the Bloch--Kato conjectures. We propose a different approach which is analogous to the approach in the proof of Theorem \ref{theorem_ell_adic_case}: consider the commutative diagram with exact rows
$$\xymatrix{
\CH^{d}(X_k,1)_{\bb Z/p^r\bb Z}\ar[r]^{} \ar[d]_{\cong} & \CH^{d+1}(X,1)_{\bb Z/p^r\bb Z}\ar[d]_{\rho_{X,p^r}^{2d+1,d+1}} \ar[r] & \CH^{d+1}(X_K,1)_{\bb Z/p^r\bb Z} \ar[d]_{\rec_{X_K}^{\bb Z/p^r\bb Z}} \ar[r] & \CH^{d}(X_k)_{\bb Z/p^r\bb Z} \ar[d]_{\cong} \ar[r]^{}  & 0 \\
 H^{2d+1}_{X_k}(X,\cal T_r(d+1))\ar[r]^{}  & H^{2d+1}(X,\cal T_r(d+1)) \ar[r]^{}  & H^{2d+1}(X_K,\mu_{p^r}^{\otimes d+1}) \ar[r]^{}  &   H^{2d+2}_{X_k}(X,\cal T_r(d+1)) & \
}$$
in which the $\cal T_r(d+1)$ are so called $p$-adic \'etale Tate twists. These are considered to be the right $p$-coefficients in mixed characteristic, at least for smooth and semi-stable schemes, and are defined by gluing $\mu_{p^r}^{\otimes d+1}$ coming from the generic fiber and $\nu_r(d)$ coming from the special fiber. For more details we refer to \cite{Sa07}. The outer vertical morphisms are isomorphisms by purity and the Kato conjectures with $p$-coefficients over finite fields. Indeed, in weight $d+1$, one has full purity for the Tate twists by \cite[Lem. 7.3.3]{Sa07} implying that $H^{2d+1}_{X_k}(X,\cal T_r(d+1))\cong H^{d-1}(X_k,\nu_r(d))$ and $H^{2d+2}_{X_k}(X,\cal T_r(d+1))\cong H^{d}(X_k,\nu_r(d))$. These latter groups are isomorphic to the respective Chow groups in the upper row by the Kato conjectures and higher class field theory respectively (see for example \cite[Thm. 9.3(ii)]{KeS12}).

In order to show that $\rec_{X_K}^{\bb Z/p^r\bb Z}$ is an isomorphism it therefore suffices to show that $\rho_{X,p^r}^{2d+1,d+1}$ is an isomorphism. 
For this we need to relate the group $\CH^{d+1}(X,1)_{\bb Z/p^r\bb Z}$ to information on the special fiber. The right objects to consider for this purpose are the groups
$$H^d(X_k,\cal K^M_{X_n,d+1}/p^r),$$
where we denote by $\cal K^M$ the sheafification of the presheaf defined in Def. \ref{definition_KM}.
These groups can be interpreted as the higher Chow groups "$\CH^{d+1}(X_n,1)_{\bb Z/p^r\bb Z}$'' of the thickened special fiber $X_n$. Note that a version of Bloch's formula tells us that $H^d(X_k,\cal K^M_{X_1,d+1}/p^r)\cong \CH^{d+1}(X_1,1)_{\bb Z/p^r\bb Z}$. The relation with $\CH^{d+1}(X,1)_{\bb Z/p^r\bb Z}$ is through the restriction map 
$$\mathrm{res}: \CH^{d+1}(X,1)_{\bb Z/p^r\bb Z}  \xto{\cong} H^d_{\mathrm{Nis}}(X,\cal K^M_{X,d+1}/p^r)\to \varprojlim_n H^d_{\mathrm{Nis}}(X_k,\cal K^M_{X_n,d+1}/p^r). $$
This definition requires the Gersten conjecture for Milnor K-theory for the isomorphism on the left (again a version of Bloch's formula) which is known in the Nisnevich topology by \cite{LuMo20}. The Nisnevich cohomology groups on the right are in most cases isomorphic to their Zariski analogues.
\begin{theorem}\label{theorem_Haas_L} (\cite[Cor. 1.9]{LudersHaas})
Let $d=\dim X_K=2.$ Then the restriction map
$$\mathrm{res}:\CH^{d+1}(X,1)_{\bb Z/p^r\bb Z}  \xto{\cong}  \varprojlim_n H^d_{\mathrm{Nis}}(X_k,\cal K^M_{X_n,d+1}/p^r) $$
is an isomorphism for all $r\geq 1$.
\end{theorem}
\begin{proof}
By \cite[Thm. 7.2]{LudersHaas}, whose proof uses Theorem \ref{theorem_JS_p}, there is an isomorphism $\CH^{d+1}(X,1)_{\bb Z/p^r\bb Z}\xto{\cong} H^{2d+1}(X,\cal T_r(d+1))$ and by \cite[Thm. 5.2]{LuedersGerstenTatePTwists} (see also \cite[Prop. 1.4]{Lueders2019}) there is an isomorphism $$\varprojlim_n H^d_{\mathrm{Nis}}(X_k,\cal K^M_{X_n,d+1}/p^r) \xto{\cong}  H^{2d+1}(X,\cal T_r(d+1)).$$ The theorem now follows from the commutative diagram
$$\xymatrix{
\CH^{d+1}(X,1)_{\bb Z/p^r\bb Z} \ar[d]_-{\cong} \ar[r]^-{\mathrm{res}} & \varprojlim_n H^d_{\mathrm{Nis}}(X_k,\cal K^M_{X_n,d+1}/p^r) \ar[d]^-{\cong}  \\
H^{2d+1}(X,\cal T_{p^r}(d)) \ar[r]^-{\cong}  & H^{2d+1}(X_k,i^*\cal T_{p^r}(d))  \
}$$
in which the lower horizontal morphism is an isomorphism by proper base change.
\end{proof}

Theorem \ref{theorem_Haas_L} uses Theorem \ref{theorem_JS_p} in an essential manner. Therefore, according to the philosophy we have been employing we would like to ask and pose the following question and conjecture:



\begin{question}
Let the notation be as above and $d=2$. Is there a geometric proof of the fact that
$$\mathrm{res}:\CH^{d+1}(X,1)_{\bb Z/p^r\bb Z} \to \varprojlim_n H^d_{\mathrm{Nis}}(X_k,\cal K^M_{X_n,d+1}/p^r)$$
is an isomorphism which works for $\roi_K$ a henselian discrete valuation ring with arbitrary residue field? 
\end{question}

\begin{conjecture}\label{conjecture2}
$\mathrm{res}$ is an isomorphism for arbitrary $d$.
\end{conjecture}


If $\roi_K=W(k)$ for a finite field $k$ of characteristic $p>2$, then there is a well-defined map 
$$\Omega^d_{X_1}\to \cal K^M_{d+1,X_n} $$
$$x\dlog y_1\wedge\dots\wedge\dlog y_d\mapsto \{1+x\pi^n,y_1,\dots,y_d\}$$
making the sequence 
$$\Omega^d_{X_1}\to \cal K^M_{d+1,X_n}\to \cal K^M_{d+1,X_{n-1}}\to 0$$ 
exact (see for example \cite[Prop. 2.5]{LuedersDeformationChow}). This sequence induces an exact cohomology sequence 
$$H^d(X_1,\Omega^d_{X_1})\to H^d(X_1,\cal K^M_{d+1,X_n})\to H^d(X_1,\cal K^M_{d+1,X_{n-1}})\to 0.$$ 
Returning to the unit filtration on $\roi_K^\times$ and noting that $\mathrm{coker}[1+\frak p^{m+1}\to 1+\frak p^{m}]\cong \ker[\roi_K^\times/1+\frak p^{m+1} \to \roi_K^\times/1+\frak p^{m}]$, we interpret the group $H^d(X_1,\Omega^d_{X_1})$ as calculating, or at least surjecting onto, the graded pieces of the unit filtration mod $p^r$ in the higher dimensional setting.

\begin{remark}
Since by Serre duality $H^d(X_1,\Omega^d_{X_1})\cong k$, Conjecture \ref{conjecture2} holds if $X_k$ is separably rationally connected by \cite[Thm. 1.1(iii)]{LuedersRatCon} if the map $k\cong H^d(X_1,\Omega^d_{X_1})\to H^d(X_1,\cal K^M_{d+1,X_n})$ is injective (which is not shown in loc. cit. but the proof should be similar to the proof of \cite[Thm. 5.2]{LuedersRatCon}).
\end{remark}

We summarise the general picture we have obtained in the following table:
\begin{table}[!h]
  \begin{adjustbox}{center,max width=\linewidth}
    \begin{tabular}{rrlrrr}
      \toprule
   Ramification & \bf $K$ & \bf $X_K$
              & \bf Graded pieces $K$ & \bf Graded pieces $X_K$ \\
      \midrule
   Unramified          & $\bb Z\to \hat{\bb Z}$         & $\CH^d(X_k)\to \pi_1^{\ab}(X_k)$                     & 0             & 0             \\
   Tamely ram.           & $U^{(0)}/U^{(1)}\to \mathrm{I}/\mathrm{V}$         & \stackbox[r]{$\CH^{d+1}(X,1)_{\bb Z/\ell\bb Z}\to$}{$H^{2d+1}(X,\mu_\ell^{\otimes d+1})$}                      & $\ell$-torsion: $k^\times=\bb F_q^\times$            & $\CH^{d+1}(X_k,1)_{\bb Z/\ell\bb Z}$             \\
    Wildly ramif.            & $U^{(1)}\to \mathrm{V}$  &  $\CH^{d+1}(X,1)_{\bb Z/p^r\bb Z}\to\varprojlim_n H^d(X_k,\cal K^M_{X_n,d+1}/p^r)$                    & $p$-torsion: $k^+=\bb F_q$            & $H^{d}(X_k,\Omega^{d}_{X_k})$             \\

      \bottomrule
    \end{tabular}
  \end{adjustbox}
  \caption{Comparison of classical and higher dimensional class field theory}
  \label{tab:PECs}
\end{table}


\bibliographystyle{acm}
\bibliography{Bibliografie}

\noindent
\parbox{0.5\linewidth}{
\noindent
Morten L\"uders \\
Universität Heidelberg\\
Mathematisches Institut \\
Im Neuenheimer Feld 205 \\
69120 Heidelberg \\
Germany\\
{\tt mlueders@mathi.uni-heidelberg.de}
}

\end{document}